\documentclass{article}

\usepackage{amsmath}
\usepackage{amsthm}
\usepackage{amssymb}
\usepackage{amscd}
\usepackage{ifthen}
\usepackage{epsfig,graphics,graphicx}
\usepackage{setspace}
\usepackage{array,tabularx}
\usepackage{url}
\usepackage[usenames]{color}
\usepackage[pdftex,colorlinks=true,urlcolor=blue,linkcolor=blue,plainpages=false,pdfpagelabels]{hyperref}
\usepackage{caption}
\newtheorem{theorem}{Theorem}[section]

\newtheorem{definition}[theorem]{Definition}
\newtheorem{proposition}[theorem]{Proposition}
\newtheorem{corollary}[theorem]{Corollary}
\newtheorem{lemma}[theorem]{Lemma}
\newtheorem{fact}[theorem]{Remark}
\newtheorem{exemplu}[theorem]{Example}
\newtheorem{exercise}{Exercise}
\newtheorem{notation}[theorem]{Notation}

\newcommand{\bdfn}{\begin{definition}}
\newcommand{\edfn}{\end{definition}}
\newcommand{\bthm}{\begin{theorem}}
\newcommand{\ethm}{\end{theorem}}
\newcommand{\bprop}{\begin{proposition}}
\newcommand{\eprop}{\end{proposition}}
\newcommand{\bcor}{\begin{corollary}}
\newcommand{\ecor}{\end{corollary}}
\newcommand{\blem}{\begin{lemma}}
\newcommand{\elem}{\end{lemma}}
\newcommand{\bfact}{\begin{fact}}
\newcommand{\efact}{\end{fact}}
\newcommand{\bex}{\begin{exemplu}\begin{rm}}
\newcommand{\eex}{\end{rm}\end{exemplu}}
\newcommand{\bxc}{\begin{exercise}}
\newcommand{\exc}{\end{exercise}}
\newcommand{\bntn}{\begin{notation}}
\newcommand{\entn}{\end{notation}}

\newcommand{\be}{\begin{enumerate}}
\newcommand{\ee}{\end{enumerate}}
\newcommand{\bce}{\begin{center}}
\newcommand{\ece}{\end{center}}
\newcommand{\bi}{\begin{itemize}}
\newcommand{\ei}{\end{itemize}}
\newcommand{\bt}{\begin{tabular}}
\newcommand{\et}{\end{tabular}}
\newcommand{\beq}{\begin{equation}}
\newcommand{\eeq}{\end{equation}}
\newcommand{\ba}{\begin{array}} 
\newcommand{\ea}{\end{array}}
\newcommand {\bea} {\begin{eqnarray}}
\newcommand {\eea} {\end {eqnarray}}
\newcommand {\bua} {\begin{eqnarray*}}
\newcommand {\eua} {\end {eqnarray*}}

\newcommand{\se}{\subseteq}

\newcommand{\ds}{\displaystyle}

\def\R{{\mathbb R}}
\def\N{{\mathbb N}}
\def\Z{{\mathbb Z}}

\def\R{{\mathbb R}}

\newcommand{\cB}{\mathcal{B}}

\newcommand{\eps}{\varepsilon}

\newcommand{\limn}{\ds\lim_{n\to\infty}}

\newcommand{\lsupn}{\ds \limsup_{n\to\infty}}

\newcounter{ct}

\newcommand{\solution}[1]{\begin{proof}#1\end{proof}}

\begin{document}

\title{Effective results on nonlinear ergodic averages in CAT$(\kappa)$ spaces}
\author{Lauren\c{t}iu Leu\c{s}tean${}^{1,2}$,  Adriana Nicolae${}^{3,4}$ \\[0.2cm]
\footnotesize ${}^1$ Faculty of Mathematics and Computer Science, University of Bucharest,\\
\footnotesize Academiei 14,  P.O. Box 010014, Bucharest, Romania\\[0.1cm]
\footnotesize ${}^2$ Simion Stoilow Institute of Mathematics of the Romanian Academy,\\
\footnotesize P. O. Box 1-764, RO-014700 Bucharest, Romania\\[0.1cm]
\footnotesize ${}^3$ Department of Mathematics, Babe\c{s}-Bolyai University, \\
\footnotesize  Kog\u{a}lniceanu 1, 400084 Cluj-Napoca, Romania\\[0.1cm]
\footnotesize ${}^4$ Simion Stoilow Institute of Mathematics of the Romanian Academy,\\
\footnotesize Research group of the project PD-3-0152,\\
\footnotesize P. O. Box 1-764, RO-014700 Bucharest, Romania\\[0.1cm]
\footnotesize E-mails:  Laurentiu.Leustean@imar.ro, anicolae@math.ubbcluj.ro
}

\date{}

\maketitle

\begin{abstract}
In this paper we apply proof mining techniques to compute, in the setting of CAT$(\kappa)$ spaces (with $\kappa >0$), effective and highly uniform rates of asymptotic 
regularity  and metastability for a nonlinear generalization of the ergodic averages, known as the Halpern iteration. 
In this way, we obtain  a uniform quantitative 
version of a nonlinear extension of  the classical von Neumann mean ergodic theorem. 
\\

\noindent {\em MSC:} 47H25,  03F10, 47J25, 47H09.\\

\noindent {\em Keywords:} Proof mining, nonlinear ergodic averages,  CAT$(\kappa)$ spaces, rates of metastabilty, Halpern iteration, asymptotic regularity.
\end{abstract}

\section{Introduction}
In this paper we apply methods from mathematical logic to obtain a uniform quantitative 
version of a generalization of 
the classical von Neumann mean ergodic theorem, giving effective rates of metastability for 
the so-called Halpern iteration, a nonlinear 
generalization of the ergodic averages. Our results are a contribution to the line of research 
known as {\em proof mining}, initiated in the 50's 
by Kreisel under the name of {\em unwinding of proofs} and extensively developed by Kohlenbach, 
beginning with the 90's. 
The idea of this research direction is to extract new, effective information from mathematical 
proofs making use of ineffective principles.  Hence, 
it can be related to Terence Tao's proposal \cite{Tao07} of {\em hard analysis}, based on 
finitary arguments, instead of the infinitary ones from {\em soft analysis}.  
Proof mining has already been applied in approximation theory, nonlinear analysis,  ergodic theory, 
topological dynamics and Ramsey theory.  Related to these  applications,  general logical metatheorems were proved,  
having the following form: if  certain statements  satisfying general logical conditions
(e.g. $\forall\exists$-sentences) are proved in some formal system associated to an abstract space, 
then uniform finitary versions of these statements are guaranteed to hold and,  furthermore, one can transform the initial proof  into a 
quantitative one for the finitary version and, in this way, extract effective uniform bounds. We refer to Kohlenbach's book \cite{Koh08-book} for an introduction to proof mining.

Our theorems guarantee under general logical conditions such strong uniform versions of non-uniform existence statements. Moreover, they provide algorithms for actually extracting effective uniform bounds and transforming the original proof into one for the stronger uniformity result. 

Let us recall the Hilbert space formulation of the celebrated von Neumann mean ergodic theorem.
\bthm
Let $H$ be a Hilbert space and $U:H\to H$ be a unitary operator. Then for all $x\in H$, the 
Ces\`{a}ro mean $ x_n=\frac1n\sum_{i=0}^{n-1}U^ix$ converges
strongly to the projection of $x$ onto the set of fixed points of $U$.
\ethm
\noindent If ${\cal X}=(X,\cB,\mu,T)$ is a probability measure-preserving system, $H=L^2({\cal X})$ 
and $U=U_T:L^2({\cal X})\to L^2({\cal X}), \, f\mapsto f\circ T$ is 
the induced operator, the Ces\`{a}ro mean  starting with $f\in L^2({\cal X})$ becomes the   
ergodic average $ A_nf=\frac1n\sum_{i=0}^{n-1}f\circ T^i$.\

The convergence of the ergodic averages can be arbitrarily slow, as shown by Krengel 
\cite{Kre78}. Furthermore, one cannot expect, in general, to get 
effective rates of convergence for the ergodic averages. Avigad, Gerhardy and Towsner 
\cite{AviGerTow10} applied methods of computable analysis on Hilbert spaces 
to obtain an example of a 
computable Lebesgue measure-preserving transformation $T$ on $[0,1]$ and a computable characteristic 
function $\chi_A$ such that the limit of the sequence $A_n\chi_A$ 
is not a computable element of $L^2([0,1])$, which implies that there is no computable bound 
on the rate of convergence of $(A_n\chi_A)$.

However, one can consider the following equivalent reformulation of the Cauchy property of $(x_n)$:
\beq
\forall k\in\N\,\forall g : \N\to\N\, \exists N\forall i,j\in[N, N + g(N)] \,\,
\left(\|x_i-x_j\|< 2^{-k}\right). \label{def-meta-xn}
\eeq
This is known in logic as Kreisel's \cite{Kre51,Kre52} no-counterexample interpretation of the Cauchy property
and it was popularized in the last years under the name of {\em metastability} by Tao \cite{Tao07,Tao08}.
In \cite{Tao08}, Tao generalized the mean ergodic theorem for multiple commuting measure-preserving 
transformations, by deducing it from a  finitary norm convergence result, 
expressed in terms of metastability. Recently, Walsh \cite{Wal12} used again metastability 
to show the  $L^2$-convergence of multiple polynomial ergodic averages arising from 
nilpotent groups of measure-preserving transformations. 

Logical metatheorems developed by Kohlenbach \cite{Koh05} show that, from wide classes of 
mathematical proofs one can extract  effective bounds on $\exists N$ in \eqref{def-meta-xn}. 
Thus, taking $\eps>0$ instead of $2^{-k}$, 
we define a {\em rate of metastability} as a functional $\Phi:(0,\infty)\times \N^{\mathbb N}\to{\mathbb N}$ satisfying
\beq
\forall \eps>0\,\forall g : 
\N\to\N\, \exists N\leq \Phi(\eps,g)\,\forall i,j\in[N, N + g(N)] \,\,\left(\|x_i-x_j\|< 
\eps\right).
\eeq

Thus, a natural direction of research is to obtain finitary, quantitative versions of convergence statements for sequences $(x_n)$ by providing effective rates of metastability. A qualitative feature of these quantitative versions is that the rates of metastability are highly uniform and independent or have only a weak dependence on the input data. Furthermore, these quantitative versions can be thereafter generalized to new structures, obtaining as an immediate consequence the generalization of the initial (non-quantitative) Cauchy statement to these structures. The main quantitative result of this paper is obtained in this way.  We refer  to \cite{KohLeu10} for another example in the context of the asymptotic behaviour of nonlinear iterations. 

Avigad, Gerhardy and Towsner \cite{AviGerTow10} computed for the first time explicit and 
uniform rates of metastability for the ergodic averages, by a logical analysis  
of Riesz' proof of the mean ergodic theorem. Their result was generalized, with better bounds, 
to uniformly convex Banach spaces by Kohlenbach and the first author \cite{KohLeu09}, 
applying proof mining methods, but this time to a proof of Garrett Birkhoff \cite{Bir39}. 
In fact, Avigad and Rute \cite{AviRut13} realized that the computations in \cite{KohLeu12}
allow one to obtain an effective bound on the number of  
$\eps$-fluctuations (i.e. pairs $(i,j)$ with $i>j$ and $\|x_i-x_j\|>\eps$). 
A very nice discussion on the different types of quantitative information (metastability, 
effective learnability, bounds on the number of oscillations) that can be extracted 
from convergence proofs is done in a recent paper by  Kohlenbach and  Safarik  \cite{KohSaf13}.

In the important paper \cite{Wit92}, Wittmann obtained the following nonlinear generalization of the mean ergodic theorem.

\begin{theorem}\label{wittmann-thm}\cite{Wit92}
Let $C$  be a bounded closed convex subset of a Hilbert space $X$, $T:C\to C$ a nonexpansive mapping and  
$(\lambda_n)_{n\ge 1}$ a sequence in $[0,1]$. For any $u\in C$, define 
\beq 
x_0=u, \quad x_{n+1}=\lambda_{n+1}u+(1-\lambda_{n+1})Tx_n. \label{intro-def-Halpern-iterate}
\eeq
Assume that $(\lambda_n)$ satisfies  
\beq
\ds \lim_{n\to\infty} \lambda_n =0, \quad  \sum_{n=1}^\infty|\lambda_{n+1}-\lambda_n| <\infty \quad \text {and} \quad 
\sum_{n=1}^\infty \lambda_n=\infty \label{hyp-lambdan}
\eeq
Then for any $u\in C$, $(x_n)$ converges to the projection $P_{Fix(T)}u$ of 
$u$ onto the (nonempty) set of fixed points $Fix(T)$.
\end{theorem} 
One can easily see that $(x_n)$ coincides with the Ces\`{a}ro mean when $T$ is linear and 
$\ds \lambda_n=\frac{1}{n+1}$. 
The iteration $(x_n)$ is known as the Halpern iteration, as it was introduced by Halpern \cite{Hal67} for the  
special case $u=0$. We refer to \cite[Section 3]{KohLeu12a} for a discussion on results 
in the literature on Halpern iterations, obtained by considering different conditions on 
$(\lambda_n)$ or more general spaces. 

Kohlenbach's logical metatheorem for Hilbert spaces \cite{Koh05} guarantees also in the case 
of Wittmann's theorem that from its proof one can extract 
a rate of metastability $\Phi$ of $(x_n)$,  uniform in the following sense: it depends only 
on $\eps$ and $g$, an upper bound on the diameter of $C$ and moduli on $(\lambda_n)$, given 
by the quantitative version of \eqref{hyp-lambdan}.
Thus, $\Phi$ is independent with respect to the starting point $u$ of the iteration, 
the nonexpansive mapping $T$, the Hilbert space $X$ and depends on $C$ only via its diameter.
Kohlenbach \cite{Koh11} computed such a uniform rate of metastability by a logical analysis of Wittmann's proof.

Furthermore, Kohlenbach and the first author  \cite{KohLeu12,KohLeu12a,KohLeu13} extracted rates 
of metastability  from the proofs of two generalizations of Wittmann's theorem given by Shioji 
and Takahashi \cite{ShiTak97} for a class of Banach spaces with a uniformly G\^ateaux differentiable 
norm  and by Saejung \cite{Sae10} for CAT$(0)$ spaces. 
Both Saejung's and Shioji-Takahashi's proofs use Banach limits (whose existence requires the 
axiom of choice), inspired by Lorentz' seminal paper \cite{Lor48}, introducing almost convergence. 
Our quantitative results were obtained 
by developing  in \cite{KohLeu12a} a method to eliminate the use of Banach limits from these 
proofs and get, in this way, 
elementary proofs to which general logical metatheorems for CAT$(0)$ spaces \cite{Koh05} and for uniformly 
smooth Banach spaces \cite{KohLeu12} can be applied to guarantee the extractability  of 
effective bounds. We point out that the use of Lorentz' almost convergence (and hence, 
Banach limits) in nonlinear ergodic theory was introduced by Reich \cite{Rei78}, while 
Bruck and Reich \cite{BruRei81} applied Banach limits for the first time to the study of 
Halpern iterations (see also \cite[Sections 12, 14]{GoeRei84}).

Geodesic spaces provide a suitable setting for extending the notion of sectional curvature 
from Riemannian manifolds. 
An important class of geodesic spaces of bounded curvature are CAT$(\kappa)$ spaces, 
where geodesic triangles are in some sense ``thin''. 
Such spaces enjoy nice properties inherited from the comparison with the model spaces and 
proved to be relevant in various problems and aspects in geometry (see \cite{BriHae99}). 

Recently, Pi\c{a}tek \cite{Pia11} extended Wittmann's result to the context of CAT$(\kappa)$ spaces with $\kappa>0$. {\em In this paper we extract an effective and uniform rate of metastability for this generalization of Wittmann's theorem.
}

Our main quantitative result (Theorem \ref{main-theorem-metastability})  is obtained by  generalizing to CAT$(\kappa)$ spaces  the quantitative proof for  CAT$(0)$ spaces  from \cite{KohLeu12a}.  
Thus, we apply  again the general method developed in \cite{KohLeu12a}, 
together with the remark  that, in fact, our logical analysis of Saejung's 
proof for CAT$(0)$ spaces results in the elimination of any contribution of Banach limits, 
hence even the finitary lemmas proved in \cite[Section 8]{KohLeu12a} are no longer needed (see \cite{KohLeu13}).
Despite this simplification, the proofs we give in this paper are much more involved, 
since we work in the setting of CAT$(\kappa)$ spaces. However, we still get a rate of metastability having a form 
similar to the one described  in \cite{KohSaf13}.  

As the first step in the convergence proof is to obtain the asymptotic regularity, our first important result 
(Proposition \ref{effective-as-reg-divergent-sum}) consists in the computation of a uniform rate of asymptotic regularity.
\mbox{}

{\em For the rest of the paper $\N=\{0,1,2,\ldots\}$ and $\Z_+=\{1,2,\ldots\}$.  Furthermore, we 
consider  CAT$(\kappa)$ spaces with $\kappa>0$. }

\section{CAT$(\kappa)$ spaces}

Let $(X,d)$ be a metric space.  A {\em geodesic path} from $x$ to $y$ is a mapping $c~:~[0,l] ~\subseteq ~\mathbb{R}~ \to ~X$ such that $c(0) = x, c(l) = y$ and 
$d\left(c(t),c(t^{\prime})\right) = \left|t - t^{\prime}\right|$ for every $t,t^{\prime} \in [0,l]$. The image $c\left([0,l]\right)$ of $c$ forms a {\em geodesic segment} 
which joins $x$ and $y$. Note that a geodesic segment from $x$ to $y$ is not necessarily unique. If no confusion arises, we use  $[x,y]$ to denote a geodesic segment 
joining $x$ and $y$.  $(X,d)$ is called a {\em (uniquely) geodesic space} if every two points $x,y \in X$ can be joined by a (unique) geodesic path. A point $z\in X$ 
belongs to the geodesic segment $[x,y]$ if and only if there exists $t\in [0,1]$ such that $d(z,x)= td(x,y)$ and $d(z,y)=(1-t)d(x,y)$, and we write $z=(1-t)x+ty$ 
for simplicity. This, too, may not be unique. A subset $C$ of $X$ is {\em convex} if $C$ contains any geodesic segment that joins every two points in $C$.  
A {\em geodesic triangle} $\Delta(x_1,x_2,x_3)$ consists of three points $x_1, x_2$ and $x_3$ in $X$ (its {\em vertices}) and three geodesic segments corresponding 
to each pair of points (its {\em edges}). 

CAT$(\kappa)$ spaces are defined in terms of comparisons with the model spaces $M^n_\kappa$. We focus here on CAT$(\kappa)$ spaces with $\kappa > 0$. We give below the precise definition and briefly describe some of their properties that play an essential role in this work. For a detailed discussion on geodesic metric spaces and, in particular, on CAT$(\kappa)$ spaces, one may check, for example, \cite{BriHae99}.

The {\em $n$-dimensional sphere $\mathbb{S}^n$} is the set $\{x \in \mathbb{R}^{n+1} : (x \ | \ x) = 1\}$, where $(\cdot \ | \ \cdot)$ stands for the Euclidean scalar product. 
Consider the mapping $d : \mathbb{S}^n \times \mathbb{S}^n \to \mathbb{R}$ by assigning to each $(x,y) \in \mathbb{S}^n \times \mathbb{S}^n$ the unique number 
$d(x,y) \in [0, \pi]$ such that $\cos d(x, y) = (x \ | \ y)$. Then, $(\mathbb{S}^n,d)$ is a metric space called the spherical space. This space is also geodesic and, 
if $d(x, y) < \pi$, then there exists a unique geodesic segment joining $x$ and $y$. Moreover, open (resp. closed) balls of radius $\le \pi/2$ (resp. $<\pi/2$) are convex. The {\it spherical law of cosines} 
states that in a spherical
triangle with vertices $x, y, z \in \mathbb{S}^n$ and $\gamma$ the spherical angle between the geodesic segments $[x, y]$ and $[x, z]$ we have
\[\cos d(y, z) = \cos d(x, y) \cos d(x, z) + \sin d(x, y) \sin d(x, z) \cos \gamma.\]

Let $\kappa > 0$ and $n \in \mathbb{N}$. The classical {\em model spaces $M^n_\kappa$} are obtained from the spherical space $\mathbb{S}^n$ by multiplying 
the spherical distance with $1/\sqrt{\kappa}$. These spaces inherit the geometrical properties from the spherical space. Thus, there is a unique geodesic path
joining $x,y \in M^n_\kappa$ if and only if $d(x,y) < \pi/\sqrt{\kappa}$. Furthermore, closed balls of radius $ < \pi/(2\sqrt{\kappa})$ are convex and we 
have a counterpart of the spherical law of cosines. We denote the {\em diameter of $M^n_\kappa$} by $D_\kappa = \pi/\sqrt{\kappa}$.

For a geodesic triangle $\Delta$=$\Delta(x_1,x_2,x_3)$, a {\em $\kappa$-comparison triangle} is a triangle $\bar{\Delta} = \Delta(\bar{x}_1, \bar{x}_2, \bar{x}_3)$ 
in $M^2_\kappa$ such that $d(x_i,x_j) = d_{M^2_\kappa}(\bar{x}_i,\bar{x}_j)$ for $i,j \in \{1,2,3\}$. For $\kappa$ fixed, $\kappa$-comparison triangles of geodesic 
triangles (having perimeter less than $2D_\kappa$) always exist and are unique up to isometry.

A geodesic triangle $\Delta$ of perimeter less than $2D_\kappa$ satisfies the {\em CAT$(\kappa)$ inequality} if for every $\kappa$-comparison triangle $\bar{\Delta}$ of $\Delta$ and for every 
$x,y \in \Delta$ we have
\[d(x,y) \le d_{M^2_\kappa}(\bar{x},\bar{y}),\]
where  $\bar{x},\bar{y} \in \bar{\Delta}$ are the comparison points of $x$ and $y$, i.e., if $x = (1-t)x_i + tx_j$ then $\bar{x} = (1-t)\bar{x}_i + t\bar{x}_j$ for $i,j \in \{1,2,3\}$. 

A metric space is called a {\em CAT$(\kappa)$ space} if every two points at distance less than $D_\kappa$ can be joined by a geodesic segment and every geodesic triangle 
having perimeter less than $2D_\kappa$ satisfies the CAT$(\kappa)$ inequality. CAT$(0)$ spaces are defined in a similar way considering the model space $M^2_0$ to be 
the Euclidean plane of infinite diameter.

\section{Main results}

Let $(X,d)$ be a geodesic space, $C\se X$ a convex subset, $T:C\to C$ a nonexpansive mapping and $(\lambda_n)$ a sequence in $[0,1]$. The Halpern iteration starting at $u\in C$ can be defined by
\beq 
x_0=u, \quad x_{n+1}=\lambda_{n+1}u + (1-\lambda_{n+1})Tx_n. \label{def-Halpern-CATk}
\eeq

The main purpose of our work is to prove a quantitative version of the following generalization of 
Wittmann's theorem to CAT$(\kappa)$ spaces, obtained recently by Pi\c{a}tek \cite{Pia11}.

\bthm\label{thm-Piatek}
Let $X$ be a complete CAT$(\kappa)$ space, $C \subseteq X$ a bounded closed convex subset with diameter 
$\ds d_C < \frac{D_\kappa}{2}$ and $T:C\to C$  a nonexpansive mapping. Assume that $(\lambda_n)$ 
satisfies \eqref{hyp-lambdan}. Then for any $u\in C$, the iteration $(x_n)$ starting from $u$ converges to the 
fixed point of $T$ which is nearest to $u$. 
\ethm

A first important result of this paper is the extraction of an effective rate of  asymptotic regularity 
for the Halpern iteration, that is, a rate of the convergence of $(d(x_n,Tx_n))$ towards $0$.  In order to state this result, 
we need to make the hypotheses \eqref{hyp-lambdan} on $(\lambda_n)$ quantitative. 

For brevity, we say that the sequence $(\lambda_n)$ and the functions $\alpha:(0,\infty)\to\Z_+$, $\gamma:(0,\infty)\to\Z_+$ and 
$\theta:\Z_+\to \Z_+$ satisfy (*) if the following conditions hold:
\be
\item $\limn \lambda_{n+1}=0$ with rate of convergence $\alpha$, i.e., 
\[\lambda_{n+1} \le \eps, \quad \text{for all } \eps>0 \text{ and all } n\geq \alpha(\eps);\] 
\item $\ds \sum_{n=1}^\infty |\lambda_{n+1}-\lambda_n|$  converges with  Cauchy modulus $\gamma$, 
i.e., 
\[\ds \sum_{i=\gamma(\eps)+1}^{\gamma(\eps)+n} |\lambda_{i+1}-\lambda_i| \leq \eps, \quad \text{for all }  
\eps>0 \text{ and all }n\in\Z_+;\]
\item $\ds\sum_{n=1}^\infty \lambda_{n+1}=\infty$ with rate of divergence $\theta$, i.e., 
\[\ds \sum_{k=1}^{\theta(n)}\lambda_{k+1} \geq n,\quad \text{for all }n\in\Z_+.\]
\ee

\bprop\label{effective-as-reg-divergent-sum}
Let $X$ be a CAT$(\kappa)$ space, $C \subseteq X$ a  bounded convex subset, 
$T:C\to C$ nonexpansive and  $M < \frac{D_\kappa}{2}$  an upper bound on the finite diameter  
$d_C$ of $C$.  Assume furthermore that $(\lambda_n),\alpha,\gamma,\theta$ satisfy (*). 

Then  $\limn d(x_n,x_{n+1})=0$ with rate of convergence $\tilde{\Phi}$ given by
\beq
\tilde{\Phi}(\eps,\kappa,M,\gamma,\theta)=\theta\left(\left\lceil\frac1{\cos(M \sqrt{\kappa})}\right\rceil\left(\gamma\left(\frac\eps{2M}\right)+
\max\left\{\left\lceil\ln\left(\frac{2M}\eps\right)\right\rceil,1\right\}\right)\right) \label{rate-xn-xn+1-sum-divergent}
\eeq
and $\limn d(x_n,Tx_n)=0$  with rate of convergence $\Phi$ given by 
\beq
\Phi(\eps,\kappa,M,\gamma,\theta,\alpha) = \max\left\{\tilde{\Phi}\left(\frac{\eps}2,\kappa,M,\gamma,\theta\right), 
\alpha\left(\frac\eps {2M}\right)\right\}.\label{rate-as-reg-sum-divergent}
\eeq
\eprop
\begin{proof}
See Section \ref{section-ef-as-reg}.
\end{proof}

If $\ds \lambda_n=\frac1{n+1}$ one can easily obtain rates $\alpha,\gamma,\theta$:
\beq
\alpha(\eps)=\gamma(\eps)=\left\lceil\frac{1}\eps\right\rceil, \qquad \theta(n)=\exp\left((n+1)\ln4\right).
\label{1n+1-alpha-beta-theta}
\eeq
As an immediate consequence we get the following:

\bcor
Assume that  $\ds \lambda_n=\frac1{n+1}, \, n\geq 1$. Then 
\[\limn d(x_n,x_{n+1})=\limn d(x_n,Tx_n)=0\]
with a common rate of convergence 
\beq
\Psi(\eps,\kappa,M) = \exp\left(\left\lceil\frac1{\cos(M \sqrt{\kappa})}\right\rceil\left\lceil \frac{8M}\eps+2\right\rceil \ln 4\right),
\eeq
which is exponential in $\ds \frac1\eps$.
\ecor

We point out that exponential rates of asymptotic regularity for the Halpern iteration were 
obtained by the first author for Banach spaces in \cite{Leu07a} and for the so-called 
$W$-hyperbolic spaces in \cite{Leu09-Habil}. Kohlenbach \cite{Koh11}  remarked that the proof 
in \cite{Leu07a} can be simplified and, as a consequence, one gets quadratic 
rates in Banach spaces. 
For CAT(0) spaces, Kohlenbach and the first author provide in \cite{KohLeu12a}  a quantitative 
asymptotic regularity result for general $(\lambda_n)$ by considering instead of 
$\ds \sum_{n=1}^\infty \lambda_{n+1}=\infty$ the equivalent condition 
$\ds  \prod_{n=1}^\infty (1-\lambda_{n+1})=0$. As a corollary, one obtains again quadratic 
rates of asymptotic regularity. However, the method used in  \cite{KohLeu12a}  for CAT(0) 
spaces does not hold for CAT$(\kappa)$ spaces.

\mbox{}

The main result of the paper is the following quantitative version of Theorem \ref{thm-Piatek}, 
which provides an explicit uniform rate of metastability for the Halpern iteration 
in CAT$(\kappa)$ spaces. 
To get such a result we  apply again the general method developed by Kohlenbach and the first author in 
\cite{KohLeu12a} for the Halpern iteration in CAT(0) spaces and applied again in 
\cite{KohLeu12} for uniformly smooth Banach spaces as well as in \cite{SchKoh12} for a modified
Halpern iteration in CAT(0) spaces. As noticed in \cite{KohLeu13}, in the end we do not need the
finitary Lemmas 8.3 and 8.4 from \cite{KohLeu12a}, since, as a consequence of the proof mining 
methods applied to Saejung's proofs, one gets a proof where no contributions of Banach limits can be traced.

\begin{theorem}\label{main-theorem-metastability}
Let $X$ be a complete CAT$(\kappa)$ space, $C \subseteq X$ a  bounded closed convex subset, 
$T:C\to C$  nonexpansive and  $M < \frac{D_\kappa}{2}$ an upper bound on the finite diameter  
$d_C$ of $C$.  Assume furthermore that $(\lambda_n),\alpha,\gamma,\theta$ satisfy (*).
Then for all $\varepsilon \in(0,2)$ and  $g:\mathbb{N} \to \mathbb{N}$,
\[
\exists N \le \Sigma(\varepsilon,g,\kappa,M,\theta,\alpha,\gamma) \ \forall m,n\in[N,N+g(N)] \ (d(x_n,x_m)\le \varepsilon),
\]
with $\ds N =\Theta_{K_0}\left(\sin^2\frac{\varepsilon\sqrt{\kappa}}{4}\right)$  for some 
$\ds \left\lceil\frac{1}{\varepsilon_0}\right\rceil\le K_0 \le \widetilde{f^*}^{B_{\eps,\kappa,M}}(0)+
\left\lceil\frac{1}{\varepsilon_0}\right\rceil$ and 
 
\[\Sigma(\varepsilon,g,\kappa,M,\theta,\alpha,\gamma)=A_{\eps,\kappa,M,\theta,\alpha,\gamma} \left(\widetilde{f^*}^{B_{\eps,\kappa,M}}(0)+
\left\lceil\frac{1}{\varepsilon_0}\right\rceil\right),\]
where the above constants and functionals are specified in Table \ref{tabel-1}.
\end{theorem}
\begin{proof}
We refer to Section \ref{section-effective-rates-meta} for the  proof. We point here only
the main steps:
\be
\item extract  a rate of asymptotic regularity (this is done in Proposition \ref{effective-as-reg-divergent-sum});
\item obtain a quantitative Browder theorem (see Proposition \ref{prop-meta-ztiu});
\item define in an appropriate way an approximate fixed point sequence $\gamma_n^t$ (see \eqref{def-gamma-nt});
\item apply Lemma \ref{quant-Aoyama-all-meta-rate-0}, a quantitative lemma on sequences of 
real numbers.
\ee
\end{proof}

\begin{table}[h!]
\begin{center}
\scalebox{0.95}{\begin{tabular}{ | l | }
\hline
\ \\
$\ds B_{\eps,\kappa,M}=\left\lceil \frac{M\sqrt{\kappa}\tan(M\sqrt{\kappa})}{1-\cos(\varepsilon_0)}\right\rceil, \quad \eps_0 = \frac{\cos(M\sqrt{\kappa})}{36}\sin^2 \frac{\varepsilon\sqrt{\kappa}}{4}$, \\
\ \\

$\ds A_{\eps,\kappa,M,\theta,\alpha,\gamma}(n)=\theta^+\left(\left\lceil\frac{1}{\cos(M\sqrt{\kappa})}\right\rceil
\left(\Gamma(n)-1+\max\left\{S,1\right\}\right)\right)+1$,\\
\ \\

$\ds \Gamma(n) =\max \left\{\chi^*_i\left(\frac{1}{3}\sin^2\frac{\varepsilon\sqrt{\kappa}}{4}\right) : 
\left\lceil\frac{1}{\varepsilon_0}\right\rceil\le i\le n \right\}, \quad \theta^+(n)= \max_{1\le i\le n} \theta(i)$,\\ 
\ \\

$\ds S=\left\lceil\ln\left(\frac{3\sin^2\frac{M\sqrt{\kappa}}4}{\sin^2 \frac{\varepsilon\sqrt{\kappa}}{4}}\right)\right\rceil, \quad \chi^*_i(\varepsilon) = \chi_i\left(\frac{\varepsilon}{2}\cos(M\sqrt{\kappa})\right), \quad L_i = \frac{\cos(M\sqrt{\kappa})\eps}{4M\sqrt{\kappa}(i+1)}$,\\ 
\ \\

$\ds \chi_i(\eps) =\max\left\{\theta\left(\left\lceil\frac1{\cos(M\sqrt{\kappa})}\right\rceil\left(\gamma(L_i)+
\max\left\{\left\lceil\ln\left(\frac{1}{L_i}\right)\right\rceil,1\right\}\right)\right),\alpha(2L_i)\right\}$,\\
\ \\

$\ds \Theta_i(\varepsilon) = \theta\left(\left\lceil\frac{1}{\cos (M\sqrt{\kappa})}\right\rceil
\left(\chi^*_i\left(\frac{\varepsilon}{3}\right)-1+
\max\left\{T,1\right\}\right)\right)+1$,\\
\ \\

$\ds T= \left\lceil\ln\left(\frac{3}{\varepsilon}\sin^2\frac{M\sqrt{\kappa}}2\right)\right\rceil, \quad \Delta^*_i(\varepsilon,g)=\frac{\varepsilon}
{3\Theta_i(\varepsilon)-3\chi^*_i(\frac{\varepsilon}{3})+3
g\left(\Theta_i(\varepsilon)\right)}$,\\ 
\ \\
 
$\ds f(i) = \max\left\{\left\lceil\frac{M\sqrt{\kappa}}{\Delta^*_i(
\sin^2\frac{\varepsilon\sqrt{\kappa}}{4},g)}\right\rceil, i\right\}-i, \quad f^*(i) = f\left(i+\left\lceil\frac{1}{\varepsilon_0}\right\rceil\right)+\left\lceil\frac{1}{\varepsilon_0}\right\rceil$\\
\ \\
and $\widetilde{f^*}(i)= i+f^*(i)$.\\
\ \\
\hline
\end{tabular}}
\caption{$\,\,$}\label{tabel-1}
\end{center}
\end{table}

Hence, we compute a rate of metastability which is uniform in the starting point $x_0$ of the iteration and the nonexpansive mapping $T$. Moreover, it depends on the space $X$ and the set $C$ only via $\kappa$ and the diameter $d_C$ of $C$. The dependence on $(\lambda_n)$ is through the rates $\alpha,\gamma,\theta$, which can be computed very easily for the natural choice $\ds\lambda_n=\frac{1}{n+1}$.  

Furthermore, as in \cite{KohLeu12,KohLeu12a} as well as in other case studies in proof mining, the rate of metastability has the form  described by Kohlenbach and Safarik  \cite{KohSaf13}. 
Thus, $g$ does not appear at all in the definition of  the mappings $A_{\eps,\kappa,M,\theta,\alpha,\gamma}$ and $B_{\eps,\kappa,M}$, and $\widetilde{f^*}(i)$ only uses $g$ on one argument, $\Theta_i(\sin^2(\eps\sqrt{\kappa}/4))$,  
which itself does not depend on $g$.  We refer to \cite{KohSaf13} for a logical explanation of this phenomenon in terms of effective learnability and bounds on the number of mind changes.

If $\ds \lambda_n=\frac1{n+1}$, with $\alpha, \gamma, \theta$ given by 
\eqref{1n+1-alpha-beta-theta}, one can easily see that $(\chi^*_i)_i$ is nondecreasing. As a consequence we obtain the following:

\bcor\label{meta-lambdan-1n+1}
Assume that $\ds \lambda_n=\frac1{n+1}$ for all $n\ge 1$. Then for all $\eps\in (0,2)$ and  $g:\N\to\N$,
\bua
\exists N\le \Sigma(\eps,g,\kappa, M)\,\,\forall m,n\in[N,N+g(N)]\,\,(d(x_n,x_m)\le \eps),
\eua
where 
\[
\Sigma(\eps,g,\kappa, M)=A_{\eps,\kappa,M} \left(\widetilde{f^*}^{B_{\eps,\kappa,M}}(0)+
\left\lceil\frac{1}{\varepsilon_0}\right\rceil\right),\]
with 
\begin{align*}
A_{\eps,\kappa,M}(n) & =\exp\left(\left(\left\lceil\frac{1}{\cos(M\sqrt{\kappa})}\right\rceil
\left(\Gamma(n)-1+\max\left\{S,1\right\}\right)+1\right)\ln 4\right)+1,\\
\Gamma(n) &=\chi^*_n\left(\frac{1}{3}\sin^2\frac{\varepsilon\sqrt{\kappa}}{4}\right),\\
\chi_i(\eps) &= \exp\left(\left(\left\lceil\frac1{\cos(M\sqrt{\kappa})}\right\rceil\left(\left\lceil\frac{1}{L_i}\right\rceil+
\max\left\{\left\lceil\ln\left(\frac{1}{L_i}\right)\right\rceil,1\right\}\right)+1\right)\ln4\right), \\
\Theta_i(\varepsilon) &= \exp\left(\left(\left\lceil\frac{1}{\cos (M\sqrt{\kappa})}\right\rceil
\left(\chi^*_i\left(\frac{\varepsilon}{3}\right)-1+
\max\left\{T,1\right\}\right)+1\right)\ln 4\right)+1
\end{align*}
and the other constants and functionals are defined as in Theorem \ref{main-theorem-metastability}.
\ecor

\section{Some technical lemmas}\label{section-CATk-technical-lemmas}

Throughout the paper, we shall use the following well-known facts:
\be
\item $x\geq \sin x$ for all $x\geq 0$.
\item $\sin(tx)\geq t\sin x$ for all $x\in[0,\pi]$ and all $t\in [0,1]$.
\item The function $f:(0,\pi)\to (0,1)$, $\ds f(x)=\frac{\sin x}x$ is decreasing.
\item Given  $t \in [0,1]$, the mapping $f:(0,\pi)\to(0,\infty)$,$ \, \ds f(x) = \frac{\sin(tx)}{\sin x}$ is increasing. 
\ee

The following very useful result is proved in \cite{Pia11} for $\kappa=1$. The proof for general $\kappa>0$ is an immediate rescaling. 
\begin{lemma}\label{CATK-ABC-convex}
Let  $\Delta(x,y,z)$ be a triangle in $X$ and $M\le \frac{D_\kappa}{2}$ 
be an upper bound on the lengths of the sides of $\Delta(x,y,z)$. Then for all  $t \in (0,1)$, 
\[
d((1 - t)x + tz,(1-t)y + tz) \le \frac{\sin \big((1-t)M\sqrt{\kappa}\big)}{\sin \big(M\sqrt{\kappa}\big)} d(x,y) \leq d(x,y).
\]
\end{lemma}

Let $X$ be a CAT$(\kappa)$ space. The next results gather some useful properties which will be needed in the subsequent sections.

\begin{lemma}\label{lemma_meta_1}
Let $\Delta(x,y,z)$ be a triangle in $X$ with perimeter  $< 2D_\kappa$. 
Let $w$ be a point on the segment joining $x$ and $z$. Suppose that $\cos (d(y,z)\sqrt{\kappa}) \ge \cos (d(y,w)\sqrt{\kappa}) \cos (d(w,z)\sqrt{\kappa})$.
 Then $d(x,w) \leq  d(x,y)$. 
Moreover, if $\Delta(\bar{x},\bar{y},\bar{z})$ is a $\kappa$-comparison triangle for $\Delta(x,y,z)$, then $\angle_{\bar{w}}(\bar{y},\bar{x}) \ge \frac{\pi}{2}$.
\end{lemma}
\solution{
Let $\Delta(\bar{x},\bar{y},\bar{z})$ be a $\kappa$-comparison triangle for $\Delta(x,y,z)$
and $\alpha = \angle_{\bar{w}}(\bar{y},\bar{z})$. Suppose that $\ds \alpha > \frac{\pi}{2}$. Then
\bua
\cos (d(y,z)\sqrt{\kappa}) &  =  & \cos (d(\bar{y},\bar{w})\sqrt{\kappa})\cos (d(\bar{w},\bar{z})\sqrt{\kappa})\\
&&  + \sin (d(\bar{y},\bar{w})\sqrt{\kappa})\sin (d(\bar{w},\bar{z})\sqrt{\kappa})\cos \alpha\\ 
& < & \cos (d(\bar{y},\bar{w})\sqrt{\kappa})\cos (d(\bar{w},\bar{z})\sqrt{\kappa})\\
& \le & \cos (d(y,w)\sqrt{\kappa})\cos (d(w,z)\sqrt{\kappa}),
\eua
which contradicts the hypothesis. Thus, $\alpha \le \frac{\pi}{2}$ and  $\ds \beta=\angle_{\bar{w}}(\bar{y},\bar{x}) \ge \frac{\pi}{2}$. It follows that 
\bua
\cos (d(\bar{x},\bar{y})\sqrt{\kappa}) 
&  =  & \cos (d(\bar{x},\bar{w})\sqrt{\kappa})\cos (d(\bar{w},\bar{y})\sqrt{\kappa})\\
&& + \sin (d(\bar{x},\bar{w})\sqrt{\kappa})\sin (d(\bar{w},\bar{y})\sqrt{\kappa})\cos \beta\\ 
& \leq  &  \cos (d(\bar{x},\bar{w})\sqrt{\kappa}),
\eua
hence $d(\bar{x},\bar{y})\geq d(\bar{x},\bar{w})$. Thus, $d(x,w) \leq d(x,y)$.
}

Assume $C \subseteq X$ is bounded with $M<\frac{D_\kappa}2$ an upper bound on its diameter. 
In the sequel $x,y,z$ are pairwise distinct points of $C$ and $w\in[x,y], \, v\in [x,z]$.

We shall use the following notation:
\begin{align*}
S_1 & = \sin (d(x,w)\sqrt{\kappa}) \sin (d(x,v)\sqrt{\kappa}), \quad  S_2 = \sin (d(x,y)\sqrt{\kappa})\sin (d(x,z)\sqrt{\kappa}),\\
S_3 & = \sin (d(x,w)\sqrt{\kappa})\sin (d(x, z)\sqrt{\kappa}), \quad  S_4  =  \sin (d(y,w)\sqrt{\kappa}) \sin (d(x, z)\sqrt{\kappa}),\\
S_5 & = \sin (d(x,w)\sqrt{\kappa})\sin (d(z,v)\sqrt{\kappa}), \\
C_1 & = \cos (d(x,w)\sqrt{\kappa}) \cos (d(x,v)\sqrt{\kappa}), \quad C_2 = \cos (d(x,y)\sqrt{\kappa}) \cos (d(x,z)\sqrt{\kappa}).
\end{align*}

\blem
\begin{align}
S_2-S_3 & \leq S_4\cos (d(x,w)\sqrt{\kappa}), \label{s2-s3}\\
S_3-S_1 & \leq S_5 \cos(d(x,v)\sqrt{\kappa}), \label{s3-s1}\\
S_2C_1-S_1C_2 & = S_4\cos (d(x, v)\sqrt{\kappa})+S_5\cos (d(x, y)\sqrt{\kappa}), \label{s2c1-s1c2}\\
S_2-S_3- S_4\cos (d(x, v)\sqrt{\kappa}) & \leq 2S_4\left(\sin^2\frac{d(x,v)\sqrt{\kappa}}2-\sin^2\frac{d(x,w)\sqrt{\kappa}}2\right),\label{s2-s3-s4}\\
S_3-S_1 - S_5\cos (d(x, y)\sqrt{\kappa}) & \leq 2S_5\left(\sin^2\frac{d(x,y)\sqrt{\kappa}}2 - \sin^2\frac{d(x,v)\sqrt{\kappa}}2\right).\label{s3-s1-s5}
\end{align}

\elem
\solution{
\begin{align*}
S_2-S_3 & = \big(\sin (d(x,y)\sqrt{\kappa})- \sin (d(x,w)\sqrt{\kappa})\big)\sin (d(x, z)\sqrt{\kappa})\\
& = 2 \sin \frac{(d(x, y)-d(x,w))\sqrt{\kappa}}{2} \cos \frac{(d(x,y)+d(x,w))\sqrt{\kappa}}2 
\sin (d(x, z)\sqrt{\kappa})\\ 
& = 2 \sin \frac{d(y,w)\sqrt{\kappa}}{2} \cos \left(\left(d(x,w)+\frac{d(y ,w)}{2}\right)\sqrt{\kappa}\right) 
\sin (d(x, z)\sqrt{\kappa})\\ 
& \le 2 \sin \frac{d(w, y)\sqrt{\kappa}}{2}\cos(d(x, w)\sqrt{\kappa})\cos\frac{d(w, y)\sqrt{\kappa}}{2}\sin (d(x, z)\sqrt{\kappa})\\
& = \sin (d(w, y)\sqrt{\kappa}) \cos (d(x,w)\sqrt{\kappa}) \sin (d(x, z)\sqrt{\kappa})= S_4 \cos (d(x,w)\sqrt{\kappa}). 
\end{align*}
Similarly, one gets that $S_3-S_1 \le S_5 \cos(d(x,v)\sqrt{\kappa})$.
\begin{align*}
S_2C_1-S_1C_2 & = \sin (d(x, y)\sqrt{\kappa})\sin (d(x, z)\sqrt{\kappa})\cos (d(x, w)\sqrt{\kappa}) \cos (d(x, v)\sqrt{\kappa})\\
& \quad - \sin (d(x, w)\sqrt{\kappa}) \sin (d(x, v)\sqrt{\kappa})\cos (d(x, y)\sqrt{\kappa}) \cos (d(x, z)\sqrt{\kappa})\\
&= \sin (d(x, z)\sqrt{\kappa})\cos (d(x, v)\sqrt{\kappa})\sin\big((d(x, y)-d(x, w))\sqrt{\kappa}\big)\\
& \quad + \sin (d(x, w)\sqrt{\kappa})\cos (d(x, y)\sqrt{\kappa})\sin\big((d(x, z)-d(x,v))\sqrt{\kappa}\big)\\
& = \sin (d(x, z)\sqrt{\kappa})\cos (d(x, v)\sqrt{\kappa})\sin(d(y,w)\sqrt{\kappa})\\
& \quad + \sin (d(x, w)\sqrt{\kappa})\cos (d(x, y)\sqrt{\kappa})\sin (d(z,v)\sqrt{\kappa})\\
& = S_4\cos (d(x, v)\sqrt{\kappa})+S_5\cos (d(x, y)\sqrt{\kappa}).
\end{align*}
Items (\ref{s2-s3-s4}) and (\ref{s3-s1-s5}) follow easily from (\ref{s2-s3}) and (\ref{s3-s1}), respectively.
}

\bprop\label{CATk-technical-ineq-1}
\beq
\sin ^2 \frac{d(w, v)\sqrt{\kappa}}{2}  \le \frac{S_1}{S_2} \sin ^2 \frac{d(y, z)\sqrt{\kappa}}{2}+\frac12(1-C_1)-\frac{S_1}{2S_2}(1-C_2).
\eeq
\eprop
\solution{
Let  $\Delta(\bar x, \bar y, \bar z)$  be a $\kappa$-comparison triangle for $\Delta(x,y,z)$. Denote $\alpha = \angle_{\bar x}(\bar y, \bar z)=\angle_{\bar x}(\bar w, \bar v)$. 
Using the cosine law we have
\begin{align*}
\cos (d(\bar w, \bar v)\sqrt{\kappa}) & = \cos (d(\bar x, \bar w)\sqrt{\kappa})\cos (d(\bar x, \bar v)\sqrt{\kappa})\\
& \quad + \sin (d(\bar x, \bar w)\sqrt{\kappa})\sin (d(\bar x, \bar v)\sqrt{\kappa}) \cos \alpha 
\end{align*}
and
\begin{align*}
\cos (d(\bar y, \bar z)\sqrt{\kappa}) & = \cos (d(\bar x, \bar y)\sqrt{\kappa})\cos (d(\bar x, \bar z)\sqrt{\kappa})\\
& \quad + \sin (d(\bar x, \bar y)\sqrt{\kappa})\sin (d(\bar x, \bar z)\sqrt{\kappa}) \cos \alpha. 
\end{align*}
Thus,
\bua
\cos (d(\bar w, \bar v)\sqrt{\kappa}) & = & \cos (d(\bar x, \bar w)\sqrt{\kappa})\cos (d(\bar x, \bar v)\sqrt{\kappa}) \\
& & + \frac{\sin (d(\bar x, \bar w)\sqrt{\kappa}) \sin (d(\bar x, \bar v)\sqrt{\kappa})}{\sin (d(\bar x, \bar y)\sqrt{\kappa}) \sin (d(\bar x, \bar z)
\sqrt{\kappa})}\bigg(\cos (d(\bar y, \bar z)\sqrt{\kappa})\\
&& -  \cos (d(\bar x, \bar y)\sqrt{\kappa})\cos (d(\bar x, \bar z)\sqrt{\kappa})\bigg)\\
& = & \frac{S_1}{S_2}\cos (d(y, z)\sqrt{\kappa})+C_1 -\frac{S_1}{S_2}C_2.
\eua
It follows that 
\bua
\frac{1 - \cos (d(w, v)\sqrt{\kappa})}{2} & \le & \frac{1}{2} + \frac{S_1}{S_2}\left(\frac{1 - \cos (d(y, z)\sqrt{\kappa})}{2} - 
\frac{1}{2}\right)- \frac{1}{2}C_1 + \frac{S_1}{2S_2}C_2.
\eua
Hence,
\bua
\sin ^2 \frac{d(w, v)\sqrt{\kappa}}{2} & \le  & \frac{S_1}{S_2} \sin ^2 \frac{d(y, z)\sqrt{\kappa}}{2}+\frac12(1-C_1)-\frac{S_1}{2S_2}(1-C_2).
\eua
}

\begin{proposition}\label{CATk-technical-ineq-2}
\be
\item 
\begin{align*}
\sin ^2 \frac{d(w, v)\sqrt{\kappa}}{2} & \leq \frac{\sin (d(x,w)\sqrt{\kappa})}{\sin (d(x,y)\sqrt{\kappa})} \sin ^2 \frac{d(y, z)\sqrt{\kappa}}{2}\\
& \quad +\frac{\sin (d(y,w)\sqrt{\kappa})}{\sin (d(x,y)\sqrt{\kappa})}\left(\sin^2\frac{d(x,v)\sqrt{\kappa}}2-\sin^2\frac{d(x,w)\sqrt{\kappa}}2\right)\\
& \quad + \frac{\sin (d(z,v)\sqrt{\kappa})}{\sin (d(x,z)\sqrt{\kappa})}\sin^2\frac{d(x,y)\sqrt{\kappa}}2.
\end{align*}
\item\label{CATk-technical-ineq-2-v2} Assume that $v=sx+(1-s)z, \, s\in [0,1]$ and $w=rx+(1-r)y, \, r\in [0,1]$. Then, 
\begin{align*}
& \sin ^2 \frac{d(w, v)\sqrt{\kappa}}{2} \leq \frac{\sin ((1-r)M\sqrt{\kappa})}{\sin (M\sqrt{\kappa})} \sin ^2 \frac{d(y, z)\sqrt{\kappa}}{2}\\
& \quad +\frac{\sin (rM\sqrt{\kappa})}{\sin (M\sqrt{\kappa})}\max\left\{\sin^2\frac{d(x,v)\sqrt{\kappa}}2-\sin^2\frac{d(x,w)\sqrt{\kappa}}2, 0\right\}\\
& \quad + \frac{\sin (sM\sqrt{\kappa})}{\sin (M\sqrt{\kappa})}\sin^2\frac{M\sqrt{\kappa}}2.
\end{align*}
\ee
\end{proposition}
\solution{
\be
\item We apply Proposition \ref{CATk-technical-ineq-1} to get that 

\begin{align*}
& \sin ^2 \frac{d(w, v)\sqrt{\kappa}}{2}  \le \frac{S_1}{S_2} \sin ^2 \frac{d(y, z)\sqrt{\kappa}}{2}+\frac12(1-C_1)-\frac{S_1}{2S_2}(1-C_2)\\
& \quad =  \frac{S_1}{S_2} \sin ^2 \frac{d(y, z)\sqrt{\kappa}}{2}\\
& \quad\quad + \frac{S_2-S_1-S_4\cos (d(x, v)\sqrt{\kappa})-S_5\cos (d(x, y)\sqrt{\kappa})}{2S_2} \\
& \quad\quad \text{by \eqref{s2c1-s1c2}}\\
& \quad \leq \frac{S_1}{S_2} \sin ^2 \frac{d(y, z)\sqrt{\kappa}}{2}+\frac{S_4}{S_2}\left(\sin^2\frac{d(x,v)\sqrt{\kappa}}2-\sin^2\frac{d(x,w)\sqrt{\kappa}}2\right)\\
& \quad\quad + \frac{S_5}{S_2}\sin^2\frac{d(x,y)\sqrt{\kappa}}2\\
& \quad\quad \text{by \eqref{s2-s3-s4} and \eqref{s3-s1-s5}},
\end{align*}
which yields the desired inequality.


\item We have that 
\begin{align*}
& \sin ^2 \frac{d(w, v)\sqrt{\kappa}}{2}  \le  \frac{\sin ((1-r)d(x,y)\sqrt{\kappa})}{\sin (d(x,y)\sqrt{\kappa})} \sin ^2 \frac{d(y, z)\sqrt{\kappa}}{2}\\
& \quad\quad +\frac{\sin (rd(x,y)\sqrt{\kappa})}{\sin (d(x,y)\sqrt{\kappa})}\max\left\{\sin^2\frac{d(x,v)\sqrt{\kappa}}2-\sin^2\frac{d(x,w)\sqrt{\kappa}}2, 0\right\}\\
& \quad\quad + \frac{\sin (sd(x,z)\sqrt{\kappa})}{\sin (d(x,z)\sqrt{\kappa})}\sin^2\frac{d(x,y)\sqrt{\kappa}}2\\
& \quad \leq   \frac{\sin ((1-r)M\sqrt{\kappa})}{\sin (M\sqrt{\kappa})} \sin ^2 \frac{d(y, z)\sqrt{\kappa}}{2}\\
& \quad\quad +\frac{\sin (rM\sqrt{\kappa})}{\sin (M\sqrt{\kappa})}\max\left\{\sin^2\frac{d(x,v)\sqrt{\kappa}}2-\sin^2\frac{d(x,w)\sqrt{\kappa}}2, 0\right\}\\
& \quad\quad + \frac{\sin (sM\sqrt{\kappa})}{\sin (M\sqrt{\kappa})}\sin^2\frac{M\sqrt{\kappa}}2.
\end{align*}
\ee
}

For the rest of the section, we assume that $v=sx+(1-s)z, \, s\in (0,1)$. We use the additional notation
\[
L_1 = \frac{S_1}{S_3}=\frac{\sin (d(x,v)\sqrt{\kappa})}{\sin (d(x, z)\sqrt{\kappa})}, \quad
L_2 = \frac{S_5}{S_3}= \frac{\sin (d(v,z)\sqrt{\kappa})}{\sin (d(x,z)\sqrt{\kappa})}.
\]

\blem\label{lemma-E}
\bea
0 < 1-L_1 \leq L_2 \cos (d(x,v)\sqrt{\kappa}), \label{1-L1-ineq} \\
\frac{L_1}{1-L_1} \leq \frac{1}{s\cos (M\sqrt{\kappa})}.\label{l1-1-l1-leq}
\eea
\elem
\solution{
\begin{align*}
& (1-L_1)\sin (d(x, z)\sqrt{\kappa}) = \sin(d(x,z)\sqrt{\kappa}) - \sin (d(x,v)\sqrt{\kappa})\\
& \quad = 2 \sin \frac{(d(x,z)-d(x,v))\sqrt{\kappa}}{2}\cos\frac{(d(x,z)+d(x,v))\sqrt{\kappa}}{2}\\
& \quad \le 2 \sin \frac{d(z,v)\sqrt{\kappa}}{2} \cos \frac{d(z,v)\sqrt{\kappa}}{2} \cos (d(x,v)\sqrt{\kappa}) \\
& \quad =  \sin(d(z,v)\sqrt{\kappa}) \cos (d(x,v)\sqrt{\kappa}).
\end{align*}
Thus, $1-L_1 \leq L_2 \cos (d(x,v)\sqrt{\kappa}).$
\begin{align*}
 \frac{L_1}{1-L_1} & =  \frac{\sin (d(x,v)\sqrt{\kappa})}{\sin (d(x,z)\sqrt{\kappa}) - \sin (d(x,v)\sqrt{\kappa})} 
\le \frac{\sin (d(x,z)\sqrt{\kappa})}{\sin (d(x,z)\sqrt{\kappa}) - \sin (d(x,v)\sqrt{\kappa})} \\
& = \frac{\sin (d(x,z)\sqrt{\kappa})}
{2\sin\frac{d(z,v)\sqrt{\kappa}}{2} \cos \left(\left(d(x,z) - \frac{d(z,v)}{2}
\right)\sqrt{\kappa}\right)}\\
& \le \frac{\sin (d(x,z)\sqrt{\kappa})}{2\sin \frac{sd(x,z)\sqrt{\kappa}}{2}\cos (d(x,z)\sqrt{\kappa})} \le  
\frac{1}{s\cos (M\sqrt{\kappa})}.
\end{align*}
}

\bprop\label{CATk-technical-ineq-3}
\be
\item 
\begin{align*}
& \sin ^2 \frac{d(y, v)\sqrt{\kappa}}{2} \le L_1 \sin ^2 \frac{d(y, z)\sqrt{\kappa}}{2}+\frac{1 - L_1}{2} -\frac12 \cos(d(x, y)\sqrt{\kappa})L_2\\
& \quad = L_2\sin^2\frac{d(x, y)\sqrt{\kappa}}2+ \frac12 (1-L_1-L_2)+ L_1 \sin ^2 \frac{d(y, z)\sqrt{\kappa}}{2}.
\end{align*}
\item\label{ineq-used-gammant} Let  $q \in C$ be such that $d(q,z)\leq d(y,v)$. 
Assume that 
\beq
\sin^2\frac{d(x, y)\sqrt{\kappa}}2 - \sin^2 \frac{d(x,v)\sqrt{\kappa}}{2} \le 0. \label{hyp-gammant}
\eeq
Then,
\bua
\sin ^2 \frac{d(y, v)\sqrt{\kappa}}{2} & \leq & 
\sin^2\frac{d(x, y)\sqrt{\kappa}}2 - \sin^2 \frac{d(x,v)\sqrt{\kappa}}{2}\\
&& + \frac{1}{s\cos (M\sqrt{\kappa})}\left(\sin^2 \frac{d(y,q)\sqrt{\kappa}}{2} + 
\sin \frac{d(y,q)\sqrt{\kappa}}{2}\right).
\eua
\ee
\eprop
\solution{
(i) We apply Proposition \ref{CATk-technical-ineq-1} with $w=y$ to get that
\begin{align*}
\sin ^2 \frac{d(y, v)\sqrt{\kappa}}{2} & \le L_1 \sin ^2 \frac{d(y, z)\sqrt{\kappa}}{2}+ \frac12(1-\cos (d(x, y)\sqrt{\kappa}) \cos (d(x, v)\sqrt{\kappa}))\\
&\quad -\frac{1}{2} L_1 (1-\cos (d(x, y)\sqrt{\kappa}) \cos (d(x, z)\sqrt{\kappa}))\\
&= L_1 \sin ^2 \frac{d(y, z)\sqrt{\kappa}}{2}+\frac{1 -L_1}{2}\\
&\quad  - \frac{\cos (d(x, y)\sqrt{\kappa})}{2\sin (d(x, z)\sqrt{\kappa})}\big( \cos (d(x, v)\sqrt{\kappa})\sin (d(x, z)\sqrt{\kappa})-\\
&\quad - \sin (d(x, v)\sqrt{\kappa}) \cos (d(x, z)\sqrt{\kappa})\big)\\
&= L_1 \sin ^2 \frac{d(y, z)\sqrt{\kappa}}{2}+\frac{1- L_1}{2} - \frac{\cos (d(x, y)\sqrt{\kappa})}{2\sin (d(x, z)\sqrt{\kappa})}\sin (d(v,z)\sqrt{\kappa})\\
&= L_1 \sin ^2 \frac{d(y, z)\sqrt{\kappa}}{2}+\frac{1 -L_1}{2}-\frac12 \cos(d(x, y)\sqrt{\kappa})L_2.
\end{align*}
(ii)
\begin{align*}
& \sin ^2 \frac{d(y, v)\sqrt{\kappa}}{2} \le L_2\sin^2\frac{d(x, y)\sqrt{\kappa}}2+ \frac12 (1-L_1-L_2)+ L_1 \sin ^2 \frac{d(y, z)\sqrt{\kappa}}{2}\\
& \quad \le  L_2\sin^2\frac{d(x, y)\sqrt{\kappa}}2+ \frac12 (1-L_1-L_2)+ L_1 \sin ^2 \frac{(d(y,q)+d(q,z))\sqrt{\kappa}}{2}\\
& \quad \le L_2\sin^2\frac{d(x, y)\sqrt{\kappa}}2+ \frac12 (1-L_1-L_2)+ L_1 \sin ^2 \frac{(d(y,q)+d(y,v))\sqrt{\kappa}}{2}\\
& \quad \le L_2\sin^2\frac{d(x, y)\sqrt{\kappa}}2+ \frac12 (1-L_1-L_2)\\
& \quad\quad +  L_1\left(\sin^2 \frac{d(y,q)\sqrt{\kappa}}{2} + \sin^2 \frac{d(y,v)\sqrt{\kappa}}{2} + \frac{1}{2}\sin (d(y,q)\sqrt{\kappa})\right)\\
& \quad\quad \text{since } \sin^2\frac{a+b}{2}\leq \sin^2 \frac{a}{2} + \sin^2 \frac{b}{2} + \frac{1}{2}\sin a \text{ for } a,b\in  \left[0,\pi\right].
\end{align*}

It follows that
\begin{align*}
\sin ^2 \frac{d(y,v)\sqrt{\kappa}}{2} \left(1 - L_1\right) & \leq  L_2\sin^2\frac{d(x, y)\sqrt{\kappa}}2+ \frac12 (1-L_1-L_2)\\
& \quad +  L_1\left(\sin^2 \frac{d(y,q)\sqrt{\kappa}}{2} + \frac{1}{2}\sin (d(y,q)\sqrt{\kappa})\right)\\
& \leq  L_2\sin^2\frac{d(x, y)\sqrt{\kappa}}2+ \frac12 (1-L_1-L_2)\\
& \quad +  L_1\left(\sin^2 \frac{d(y,q)\sqrt{\kappa}}{2} + \sin \frac{d(y,q)\sqrt{\kappa}}{2}\right)\\
& \leq  L_2\sin^2\frac{d(x, y)\sqrt{\kappa}}2-\frac12L_2(1-\cos (d(x,v)\sqrt{\kappa}))\\
& \quad  +  L_1\left(\sin^2 \frac{d(y,q)\sqrt{\kappa}}{2} + \sin \frac{d(y,q)\sqrt{\kappa}}{2}\right)\\
& \quad \text{ by \eqref{1-L1-ineq}}.
\end{align*}
Thus,
\bua
\sin^2 \frac{d(y,v)\sqrt{\kappa}}{2} & \leq &  \frac{L_2}{1-L_1}\left(
\sin^2\frac{d(x, y)\sqrt{\kappa}}2-\sin^2\frac{d(x,v)\sqrt{\kappa}}2\right)\\
&& +  \frac{L_1}{1-L_1}\left(\sin^2 \frac{d(y,q)\sqrt{\kappa}}{2} + \sin \frac{d(y,q)\sqrt{\kappa}}{2}\right).
\eua
By assumption, we have that $\ds \sin^2\frac{d(x, y)\sqrt{\kappa}}2 - \sin^2 \frac{d(x,v)\sqrt{\kappa}}{2} \le 0$.
Using the fact that $\ds \frac{L_2}{1-L_1}\geq 1$ and (\ref{l1-1-l1-leq}), it follows that
\bua
\sin^2 \frac{d(y,v)\sqrt{\kappa}}{2} 
& \leq & \sin^2\frac{d(x, y)\sqrt{\kappa}}2 - \sin^2 \frac{d(x,v)\sqrt{\kappa}}{2}\\
&& + \frac{1}{s\cos (M\sqrt{\kappa})}\left(\sin^2 \frac{d(y,q)\sqrt{\kappa}}{2} + 
\sin \frac{d(y,q)\sqrt{\kappa}}{2}\right).
\eua
}

\section{Effective rates of asymptotic regularity}\label{section-ef-as-reg}

We assume the hypothesis of Proposition \ref{effective-as-reg-divergent-sum}. 
As in \cite{Leu07a,KohLeu12,KohLeu12a}, the main tool in obtaining rates of asymptotic regularity is the following quantitative lemma, which is a slight reformulation of \cite[Lemma 1]{KohLeu12a}.

\blem\label{main-lema-as-reg-Halpern-divergent-sum}
Let $(\alpha_n)_{n\ge 1}$ be a sequence in $[0,1]$ and $(a_n)_{n\geq 1},(b_n)_{n\geq 1}$ be sequences in $\R_+$  such that
\beq
a_{n+1}\leq (1-\alpha_{n+1}) a_n + b_n \quad \text{for all~} n\in\Z_+.
\eeq
Assume that $\ds \sum_{n=1}^\infty b_n$  is convergent with Cauchy modulus  $\gamma$ and $\ds \sum_{n=1}^\infty \alpha_{n+1}$ diverges with rate of divergence $\theta$.

Then, $\ds\limn a_n=0$ with rate of convergence $\Sigma$ given by 
\beq
\Sigma(\eps,P,\gamma,\theta)=\theta\left(\gamma\left(\frac\eps 2\right)+
\max\left\{\left\lceil\ln\left(\frac{2P}\eps\right)\right\rceil,1\right\}\right)+1,
\eeq
where $P>0$ is an upper bound on  $(a_n)$.
\elem

A second useful result, which is also needed  in the  metastability proof, is the following:

\blem\label{mun-divergent-lambdan}
For all $n\geq 1$, let
\beq
\mu_n = 1- \frac{\sin\left((1-\lambda_n)M\sqrt{\kappa}\right)}{\sin(M\sqrt{\kappa})}\in (0,1). \label{def-mun}
\eeq
Then 
\be
\item\label{ineq-mun-lambda-n} $\mu_n \geq \lambda_n \cos \left(M \sqrt{\kappa}\right)$  for all $n\ge 1$. 
\item\label{sum-lambdan-mun}  $\ds\sum_{n=1}^\infty \lambda_{n+1}=\infty$  with rate of divergence $\theta$ yields  $\ds\sum_{n=1}^\infty \mu_{n+1}=\infty$  
with rate of divergence $\ds \tilde{\theta}(n) = \theta\left(\left\lceil\frac{1}{ \cos(M \sqrt{\kappa})}\right\rceil n \right)$.
\ee
\elem
\solution{
\be
\item 
One has 
\bua \mu_n 
&= & \frac{2\sin \frac{\lambda_n M \sqrt{\kappa}}2\cos 
\frac{(2-\lambda_n)M \sqrt{\kappa}}2}{\sin \left(M \sqrt{\kappa}\right)}
\ge  \frac{2\sin \frac{\lambda_n M \sqrt{\kappa}}{2}
\cos \left(M \sqrt{\kappa}\right)}{\sin \left(M \sqrt{\kappa}\right)} \\
& \ge &  \lambda_n \cos \left(M \sqrt{\kappa}\right).
\eua
\item Follows immediately from \eqref{ineq-mun-lambda-n}.
\ee
}

\blem
For all $n\in\Z_+$
\beq
d(x_n,x_{n+1}) \leq (1-\mu_{n+1})d(x_{n-1},x_n)+M|\lambda_{n + 1} - \lambda_n|.
\label{ineq-as-reg-C-bounded-M}
\eeq
\elem
\begin{proof}
Let us denote for simplicity $u_n = \lambda_{n+1} u + (1-\lambda_{n+1}) Tx_{n-1}$. Then,
\bua
d(x_n,u_n) &=& |\lambda_{n + 1} - \lambda_n|d(u,Tx_{n-1})\leq M|\lambda_{n + 1} - \lambda_n| \quad\text{and}\\
d(u_n,x_{n+1}) &\leq & \frac{\sin \big((1-\lambda_{n+1})M\sqrt{\kappa}\big)}{\sin \big(M\sqrt{\kappa}\big)}d(x_{n-1},x_n)\quad \text{by Lemma \ref{CATK-ABC-convex}}.
\eua
\end{proof}

\subsection{Proof of Proposition \ref{effective-as-reg-divergent-sum}}

Let $\tilde{\Phi}$, $\Phi$ be given by  \eqref{rate-xn-xn+1-sum-divergent} and \eqref{rate-as-reg-sum-divergent}. 
Apply Lemma \ref{main-lema-as-reg-Halpern-divergent-sum} with 
\[a_n = d(x_n,x_{n-1}),\quad b_n = M|\lambda_{n+1}-\lambda_n|\quad \text{and}  \quad \alpha_n=\mu_n, \]
and use Lemma \ref{mun-divergent-lambdan}.\eqref{sum-lambdan-mun} and the fact that $\ds \sum_{n=1}^\infty b_n$ is convergent with Cauchy modulus 
$\ds \tilde{\gamma}(\varepsilon)=\gamma\left(\frac{\varepsilon}{M}\right)$ 
to conclude that  $\limn d(x_n,x_{n+1})=0$ with rate of convergence $\tilde{\Phi}$. 

Since $d(x_n,Tx_n) \le  d(x_n,x_{n+1}) + M\lambda_{n+1}$ for all $n\geq 1$, it follows easily that $\Phi$ is a rate of asymptotic regularity. $\hfill\Box$

\section{A quantitative Browder theorem}

Let $X$ be a complete CAT$(\kappa)$ space, $C \subseteq X$  a bounded closed convex subset with diameter 
$\ds d_C < \frac{D_\kappa}{2}$ and $T:C\to C$ be nonexpansive. 

A very important step in the convergence proof for Halpern iterations is the  construction 
of a sequence of approximants converging strongly  to a fixed point of $T$. 
Given $t\in (0,1)$ and $u\in C$, Lemma \ref{CATK-ABC-convex} yields that the mapping
\beq
T_t^u:C\to C, \quad T_t^u(y)=tu+(1-t)Ty \label{averaged-u}
\eeq
is a contraction, hence it has a unique fixed point $z_t^u\in C$. Thus,
\beq
z_t^u=tu+(1-t)Tz_t^u.\label{Halpern-continuous}
\eeq

Pi\c{a}tek \cite{Pia11} obtained the following generalization  to 
CAT$(\kappa)$ spaces of an essential result due to Browder \cite{Bro66,Bro67}.

\bthm \cite{Pia11} \label{Piatek-theorem-Browder}
In the above hypothesis, $\ds\lim_{t\to 0^+}z_t^u$ exists and is a fixed point of $T$.
\ethm

In the setting of Hilbert spaces, Browder proved the result  using  weak sequential compactness and a projection argument (to the set of fixed points of $T$).  A new and elementary proof of Browder's result was given by Halpern \cite{Hal67} when $C$ is the closed unit ball and the starting 
point is $u=0$. Generalizations of Browder's theorem were obtained by 
Reich \cite{Rei80} for  uniformly smooth Banach spaces, Goebel and Reich  \cite{GoeRei84} 
for the Hilbert ball  and Kirk  \cite{Kir03} for CAT(0) spaces.

Kohlenbach \cite{Koh11} applied proof mining methods to both Browder's original proof and 
the extension of Halpern's proof to  bounded closed convex $C$ 
and arbitrary $u\in C$,  obtaining in this way quantitative versions of Browder's theorem 
with uniform effective rates of metastability. 
As pointed out in \cite[Remark 1.4]{Koh11}, one cannot expect in general to get effective 
rates of convergence. Since Kirk's proof of the generalization of Browder's 
theorem to CAT(0) spaces is obtained by a slight change of Halpern's argument, Kohlenbach's 
quantitative result
goes through basically unchanged to CAT(0) spaces (see \cite[Proposition 9.3]{KohLeu12a}).

In this section we obtain a quantitative version of Theorem \ref{Piatek-theorem-Browder}. 
As a consequence of Halpern's proof, for any nonincreasing sequence $(t_n)$ in 
$(0,1)$, one gets that $(z_{t_n}^u)$ converges strongly to some point $z\in C$, which is a 
fixed point of $T$ if $\limn t_n=0$.  Our quantitative result gives rates of 
metastability for such sequences $(z_{t_n}^u)$ and this suffices for the proof of our main 
Theorem \ref{main-theorem-metastability}.

\begin{proposition}\label{prop-meta-ztiu}
Let $X$ be a complete CAT$(\kappa)$ space, $C \subseteq X$  bounded closed convex with diameter 
$\ds d_C < \frac{D_\kappa}{2}$ and $T:C\to C$ be nonexpansive. 
Assume that $(t_n) \subseteq (0,1)$ is nonincreasing. Then for every $\varepsilon \in (0,1)$ and 
$g : \mathbb{N} \to \mathbb{N}$,
\[\exists K_0 \le K(\varepsilon, g, M) \ \forall i,j \in [K_0, K_0 + g(K_0)] \ \big(d(z_{t_i}^u, z_{t_j}^u) \le \frac{\varepsilon}{\sqrt{\kappa}}\big),\]
where  
\[\ds K(\varepsilon, g, M) = \widetilde g^{\ds \left(\left\lceil \frac{M\sqrt{\kappa}\tan(M\sqrt{\kappa})}{1 - 
\cos \varepsilon}\right\rceil\right)}(0),\] 
with $\ds d_C\leq M< \frac{D_\kappa}{2}$ and   $\widetilde g(n) = n + g(n)$.
\end{proposition}
\solution{ 
Let $\varepsilon \in (0,1)$ and $g:\N\to\N$. We assume without loss of generality that $i < j$, hence $t_j \le t_i$. 
Denote $\ds u_{i,j} = t_j u + (1 - t_j) T z_{t_i}^u$.
Then,
\[d(u, z_{t_i}^u) = (1 - t_i)d(u, T z_{t_i}^u) \le (1 - t_j)d(u, T z_{t_i}^u) = d(u, u_{i,j}), \]
so $z_{t_i}^u \in [u, u_{i,j}]$. It follows by  Lemma \ref{CATK-ABC-convex} that
$d(z_{t_j}^u, u_{i,j}) \le d(T z_{t_j}^u, T z_{t_i}^u) \le d(z_{t_j}^u, z_{t_i}^u)$. 

We can apply now Lemma \ref{lemma_meta_1} with $x=u, y=z_{t_j}^u, z=u_{i,j}, w=z_{t_i}^u$  
to get that $d(u, z_{t_i}^u) \le d(u, z_{t_j}^u)$ and for 
$\Delta(\bar u, \bar z_{t_j}^u, \bar u_{i,j})$ a $\kappa$-comparison triangle of 
$\Delta(u, z_{t_j}^u, u_{i,j})$, one has 
$\angle_{\bar z_{t_i}^u}(\bar u, \bar z_{t_j}^u) \ge \pi/2$. Since $(d(u, z_{t_n}^u))_n$  
is a nondecreasing sequence in $[0,M]$, by an application of  \cite[Lemma 4.1]{Koh11}, there exists $K_0 \le K(\varepsilon, g, M)$ such that
\bua
\forall i,j \in [K_0, K_0 + g(K_0)] \left(|d(u,z_{t_j}^u) - d(u, z_{t_i}^u)| 
\le \frac{1 - \cos\varepsilon}{\sqrt{\kappa}\tan (M\sqrt{\kappa})}\right).
\eua

Let now $i<j \in [K_0, K_0 + g(K_0)]$. 
Then, 
\bua
\cos (d(u,z_{t_i}^u)\sqrt{\kappa}) - \cos (d(u, z_{t_j}^u)\sqrt{\kappa})
&\le & \sin(M\sqrt{\kappa}) \frac{1 - \cos\varepsilon}{\tan (M\sqrt{\kappa})} \\
&= & (1 - \cos\varepsilon)\cos(M\sqrt{\kappa}).
\eua
Furthermore, by the cosine law and the fact that $\angle_{\bar z_{t_i}^u}(\bar u, \bar z_{t_j}^u)\geq \frac{\pi}2$, we have that
\bua
\cos (d(\bar u, \bar z_{t_j}^u)\sqrt{\kappa})
&\leq & \cos (d(\bar u, \bar z_{t_i}^u)\sqrt{\kappa})\cos (d(\bar z_{t_i}^u, \bar z_{t_j}^u)\sqrt{\kappa}).
\eua
It follows that 
\begin{align*}
&\cos (d(u, z_{t_i}^u)\sqrt{\kappa}) - (1 - \cos \varepsilon)\cos (M\sqrt{\kappa})\\
& \quad \le \cos (d(u, z_{t_j}^u)\sqrt{\kappa}) = \cos (d(\bar u, \bar z_{t_j}^u)\sqrt{\kappa})\\ 
& \quad \le \cos (d(\bar u, \bar z_{t_i}^u)\sqrt{\kappa})\cos (d(\bar z_{t_j}^u, \bar z_{t_i}^u)\sqrt{\kappa})\\ 
& \quad \le \cos (d(u, z_{t_i}^u)\sqrt{\kappa})\cos (d(z_{t_j}^u, z_{t_i}^u)\sqrt{\kappa}).
\end{align*}
Hence
\begin{align*}
\cos (d(z_{t_j}^u, z_{t_i}^u)\sqrt{\kappa}) \ge 
1 - (1 - \cos\varepsilon)\frac{\cos (M\sqrt{\kappa})}{\cos (d(u, z_{t_i}^u)\sqrt{\kappa})}
\ge \cos\varepsilon.
\end{align*}
Thus, $d(z_{t_j}^u, z_{t_i}^u)\sqrt{\kappa}\leq \eps$
and the proof is complete.
}

\section{Effective rates of metastability}\label{section-effective-rates-meta}

In this section we shall prove  the main result of our paper, Theorem \ref{main-theorem-metastability}, 
hence we assume that its hypotheses are satisfied. We give first some technical results that will be needed in the proof.

\subsection{Some useful lemmas}

As in  \cite{KohLeu12,KohLeu12a}, one of the main ingredients of our proof
is a sequence obtained by combining the Halpern iteration 
$(x_n)$  and the points $\ds z_t^u$. However, in the setting of CAT$(\kappa)$ spaces, 
its definition and the proofs of the necessary properties are based on the much more involved
technical lemmas from Section \ref{section-CATk-technical-lemmas}.

If $(a_n)$ is a real sequence, we say that $\ds \lsupn a_n\le 0$ with effective rate $\Psi:(0,\infty)\to\Z_+$ if
\[
\forall \eps>0\,\forall n\ge \Psi(\eps)\,\, (a_n\le \eps).
\]

Let us define
\beq
\gamma_n^t = \sin^2 \frac{d(u,z_t^u)\sqrt{\kappa}}{2} - \sin^2\frac{d(u,x_{n+1})\sqrt{\kappa}}{2}.
\label{def-gamma-nt}
\eeq

\bprop\label{prop-gammant-limsup-0}
\be
\item\label{item-ineq-gammant-ant-sin} For $n \ge 1$, if $\gamma_n^t \ge 0$, then
\beq
\gamma_n^t \leq \frac{a_n}t - \sin ^2 \frac{d(x_{n+1}, z_t^u)\sqrt{\kappa}}{2},
\eeq
where 
\beq
\!\!\!\!\!\!\!\!\!\!\!\!\!\!  a_n = \frac{1}{\cos (M\sqrt{\kappa})}\left(\sin^2 \frac{d(x_{n+1},Tx_{n+1})\sqrt{\kappa}}{2} + 
\sin \frac{d(x_{n+1},Tx_{n+1})\sqrt{\kappa}}{2}\right). 
\eeq 
\item \label{item-ineq-gammant-ant} $\ds \gamma_n^t \leq \frac{a_n}t$ for all $n\ge 1$.
\item\label{item-limsup-gamma-nt-effective}  $\ds \lsupn \gamma_n^t \leq 0 $ with effective 
rate $\Psi(\eps,\kappa,M,t,\gamma,\theta,\alpha)$ given  by 
\beq
\!\!\!\!\!\!\!\!\!\!\!\!\!\! \Psi=\max\left\{\theta\left(\left\lceil\frac1{\cos(M \sqrt{\kappa})}\right\rceil\left(\gamma(L)+
\max\left\{\left\lceil\ln\left(\frac1L\right)\right\rceil,1\right\}\right)\right),\alpha(2L)\right\}, 
\label{lsup-gamma-nt-effective}
\eeq
where $\ds L=\frac{\cos(M\sqrt{\kappa})t\eps}{4M\sqrt{\kappa}}$.
\item \label{item-ineq-main} For $n \ge 1$,
\begin{align*}
\sin ^2 \frac{d(x_{n+1},z_t^u)\sqrt{\kappa}}{2} &\le 
\frac{\sin \left((1-\lambda_{n+1})M\sqrt{\kappa}\right)}{\sin (M\sqrt{\kappa})}\sin ^2 \frac{d(x_n,z_t^u)\sqrt{\kappa}}{2}\\  
& \quad + \frac{\sin \left(\lambda_{n+1}M\sqrt{\kappa}\right)}{\sin (M\sqrt{\kappa})}\max\{\gamma_n^t, 0\}\\
& \quad + \frac{\sin \left(t M \sqrt{\kappa}\right)}{\sin (M\sqrt{\kappa})}\sin^2 \frac{M\sqrt{\kappa}}{2}.
\end{align*}
\ee
\eprop
\solution{
\be
\item Apply Proposition \ref{CATk-technical-ineq-3}.\eqref{ineq-used-gammant} with 
$x=u, y=x_{n+1}, z=Tz_t^u,v=z_t^u, s=t$, $q=Tx_{n+1}$ and note that 
\[
\sin^2\frac{d(u,x_{n+1})\sqrt{\kappa}}2 - \sin^2\frac{d(u,z_t^u)\sqrt{\kappa}}{2}= - \gamma_n^t\le 0.
\]
It follows that $\ds\sin ^2 \frac{d(x_{n+1}, z_t^u)\sqrt{\kappa}}{2} \leq   - \gamma_n^t+\frac{a_n}t$, 
hence \eqref{item-ineq-gammant-ant-sin}.
\item Obviously, since $\ds \frac{a_n}t\geq 0$.
\item Since $\ds a_n\leq  \frac{1}{\cos (M\sqrt{\kappa})}d(x_{n+1}, Tx_{n+1})\sqrt{\kappa}$ and, 
by Proposition \ref{effective-as-reg-divergent-sum}, the sequence
 $(d(x_n,Tx_n))$ converges to $0$ with rate of convergence $\Phi$ given by \eqref{rate-as-reg-sum-divergent}, 
 we get that $\lsupn \gamma_n^t \leq 0 $ with effective rate
\[
\Psi(\eps,\kappa,M,t,\gamma,\theta,\alpha)=\Phi\left(\frac{\cos(M\sqrt{\kappa})t\eps}{\sqrt{\kappa}}, 
\kappa,M,\gamma,\theta,\alpha\right).
\]
\item By Proposition \ref{CATk-technical-ineq-2}.\eqref{CATk-technical-ineq-2-v2} with 
$x=u, y=Tx_n, z=Tz_t^u, w=x_{n+1}, v= z_t^u, r=\lambda_{n+1}$ and $s=t$.
\ee
}

In fact, it suffices for the proof of  the main theorem to consider the case 
$\ds t_i=\frac{1}{i+1}, i\geq 0$. Then $(t_i)$ converges towards $0$ with rate  
$\ds \left\lceil\frac{1}{\eps}\right\rceil.$

We shall denote $\gamma_n^{t_i}$ with $\gamma_n^i$. Furthermore, $z_{t_i}^u$ will be 
simply denoted  by $z_i^u$. Thus, 
\beq
\gamma_n^i = \sin^2 \frac{d(u,z_i^u)\sqrt{\kappa}}{2} - \sin^2\frac{d(u,x_{n+1})\sqrt{\kappa}}{2}.
\eeq

\blem \label{lemma-Bwd-2}
Assume that $i,j\geq 0$ and $\delta \in (0,1)$  are such that  $d(u,T z_i^u) - d(u, T z_j^u) \leq  
\frac{\delta}{\sqrt{\kappa}}$. Then,
\bua
\gamma_n^i &\le&  \gamma_n^j + \sin^2 \frac{M\sqrt{\kappa}}{2(j+1)} + 2\sin \frac{M\sqrt{\kappa}}{2(j+1)} + 
\sin^2 \frac{\delta}{2} + 2\sin \frac{\delta}{2} \sin \frac{M\sqrt{\kappa}}{2}.
\eua
\elem
\solution{
We have that 
\bua
\gamma_n^i &=& \sin ^2 \frac{\frac{i}{i+1}d(u,T z_i^u)\sqrt{\kappa}}{2} - \sin^2\frac{d(u,x_{n+1})\sqrt{\kappa}}{2}\\
 &\le& \sin ^2 \frac{d(u,T z_i^u)\sqrt{\kappa}}{2} - \sin^2\frac{d(u,x_{n+1})\sqrt{\kappa}}{2}\\
& \le & \left(\sin\frac{d(u, T z_j^u)\sqrt{\kappa}}{2}+\sin\frac{\delta}{2} \right)^2 - \sin^2\frac{d(u,x_{n+1})\sqrt{\kappa}}{2}\\
& \le &\sin^2 \frac{d(u, T z_j^u)\sqrt{\kappa}}{2} - \sin^2\frac{d(u,x_{n+1})\sqrt{\kappa}}{2} + \sin^2 \frac{\delta}{2} + 2\sin \frac{\delta}{2} \sin \frac{M\sqrt{\kappa}}{2}.
\eua
Note that
\bua
\sin^2 \frac{d(u, T z_j^u)\sqrt{\kappa}}{2}  & =& \sin^2 \frac{\frac{j}{j+1}d(u, T z_j^u)\sqrt{\kappa} + \frac{1}{j+1}d(u, T z_j^u)\sqrt{\kappa}}{2}\\
&\le& \sin^2 \frac{\frac{j}{j+1}d(u, T z_j^u)\sqrt{\kappa}}{2} + \sin^2 \frac{M\sqrt{\kappa}}{2(j+1)} + 2\sin \frac{M\sqrt{\kappa}}{2(j+1)}. 
\eua

}

Finally, let us recall the following slight reformulation of \cite[Lemma 5.2]{KohLeu12a}.

\blem\label{quant-Aoyama-all-meta-rate-0}
Let $\eps \in (0,2)$, 
$g:\N\to\N$, $L>0$, $\theta:\Z_+\to\Z_+$ and $\psi:(0,\infty)\to\Z_+$. Define
\bua
\Theta=\Theta(\eps, L,\theta,\psi) & = & \theta\left(\psi\left(\frac{\eps}3\right)-1+\max\left\{\left\lceil\ln\left(\frac{3L}{\eps}\right)
\right\rceil,1\right\}\right)+1, \\
\Delta=\Delta(\eps,g, L,\theta,\psi) &= &  \frac{\eps}{3g_\eps(\Theta-\psi(\eps/3))},
\eua
where  $g_\eps(n)=n+g(n+\psi(\eps/3))$.\\
Assume that 
\be
\item $(\alpha_n)$ is a sequence in $[0,1]$ such that the series $\ds \sum_{n=1}^\infty \alpha_n$ diverges with  rate of divergence $\theta$;
\item $(t_n)$ is a sequence of real numbers such that $\ds t_n\le \frac{\eps}3$ for all $\ds n\ge \psi\left(\frac{\eps}3\right)$.
\ee
Let  $(s_n)$ be a bounded sequence of real numbers with upper bound $L$ satisfying
\beq
s_{n+1}\le (1-\alpha_n)s_n+\alpha_nt_n+\Delta \quad\quad\text{for all }n\ge 1.\label{sn-recurrence-0}
\eeq
Then $s_n\le \eps\,$ for all $n\in[\Theta,\Theta+g(\Theta)]$.
\elem

\subsection{Proof of Theorem \ref{main-theorem-metastability}}

Let $\eps\in(0,2)$ and $g:\N\to\N$ be fixed. For simplicity, we omit parameters $\kappa,M, \Phi, \theta,\alpha,\beta$ for all functionals in this proof. 
Let us define $h : (0,1) \to \mathbb{R}_+$ by
\beq
h(\delta)=\sin \frac{\delta}{2}\left(\sin\frac{\delta}{2} + 2 \sin \frac{M\sqrt{\kappa}}{2}\right) + \sin \frac{\delta M\sqrt{\kappa}}{2}
\left(\sin \frac{\delta M\sqrt{\kappa}}{2} + 2\right) \le 6 \delta.
\eeq
Take $\ds \eps_0 = \frac{\cos(M\sqrt{\kappa})}{36}\sin^2 \frac{\varepsilon\sqrt{\kappa}}{4}$.
Then, $\ds h(\eps_0) \le \frac{\cos(M\sqrt{\kappa})}{6}\sin^2 \frac{\varepsilon\sqrt{\kappa}}{4}$.

Applying Proposition \ref{prop-meta-ztiu} for $\ds t_i=\frac{1}{i+1}, \varepsilon_0$ and $f^*$, we get the existence of
 $$\ds K_1 \le K(\varepsilon_0, f^*)= \widetilde {f^*}^{\ds \left(\left\lceil \frac{M\sqrt{\kappa}\tan(M\sqrt{\kappa})}{1 - 
\cos\varepsilon_0}\right\rceil\right)}(0)$$ 
such that $\ds d(z_i^u,z_j^u)\le \frac{\varepsilon_0}{\sqrt{\kappa}}$ for  all $i,j \in [K_1, \widetilde{f^*}(K_1)]$.

Let $\ds K_0= K_1 + \left\lceil\frac{1}{\varepsilon_0}\right\rceil$ and $\ds J=K_0 + f(K_0)=\widetilde{f^*}(K_1)$.
It follows that $\ds d(z_J^u,z_{K_0}^u)\leq \frac{\varepsilon_0}{\sqrt{\kappa}}$, hence
\bua
d(u,T z_J^u) & \le & d(u, T z_{K_0}^u) + d(T z_{K_0}^u, T z_J^u) \le  d(u, T z_{K_0}^u) + d(z_{K_0}^u, z_J^u)\\
&\le& d(u, T z_{K_0}^u) + \frac{\varepsilon_0}{\sqrt{\kappa}}.
\eua

An application of Lemma \ref{lemma-Bwd-2} with $i=J, \,j=K_0$ and $\delta=\eps_0$ gives us
\bua
\gamma_n^J & \leq &  \gamma_n^{K_0}+  \sin^2 \frac{M\sqrt{\kappa}}{2(K_0+1)} + 2\sin \frac{M\sqrt{\kappa}}{2(K_0+1)} + \sin^2 \frac{\eps_0}{2} + 2\sin \frac{\eps_0}{2} \sin \frac{M\sqrt{\kappa}}{2}\\
&\le& \gamma_n^{K_0} + \sin^2 \frac{\varepsilon_0}{2} + 2\sin \frac{\varepsilon_0}{2} \sin \frac{M\sqrt{\kappa}}{2} + \sin^2 \frac{M\varepsilon_0\sqrt{\kappa}}{2} + 2\sin \frac{M\varepsilon_0\sqrt{\kappa}}{2}\\
& = &  \gamma_n^{K_0} +h(\eps_0)\leq \gamma_n^{K_0} + \frac{\cos (M\sqrt{\kappa})}{6}\sin^2 \frac{\varepsilon\sqrt{\kappa}}{4}.
\eua

Applying now Proposition \ref{prop-gammant-limsup-0}.\eqref{item-ineq-main} with $t=\frac{1}{J+1}$ and recalling the definition \eqref{def-mun} of $(\mu_n)$, 
it follows that for all $n\geq 1$,
\bua
\sin^2 \frac{d(x_{n+1}, z_J^u)\sqrt{\kappa}}{2} & \le & (1-\mu_{n+1})\sin^2 \frac{d(x_n,z_J^u)\sqrt{\kappa}}{2}\\
&& + \frac{\sin(\lambda_{n+1}M\sqrt{\kappa})}{\sin (M\sqrt{\kappa})}\max\{\gamma_n^J,0\}\\
&& + \frac{\sin\left(\frac{1}{J+1}M\sqrt{\kappa}\right)}{\sin (M\sqrt{\kappa})}\sin^2\frac{M\sqrt{\kappa}}{2}.
\eua

Since $\ds J=K_0+f(K_0)\ge \left\lceil\frac{M\sqrt{\kappa}}{\Delta^*_{K_0}(\sin^2\frac{\varepsilon\sqrt{\kappa}}{4},g)}\right\rceil$ and
\[\cos(M\sqrt{\kappa})\sin(\lambda_{n+1}M\sqrt{\kappa}) \le \sin (M\sqrt{\kappa}) - 
\sin\left((1 - \lambda_{n+1})M\sqrt{\kappa}\right),\]
it follows that 
\bua \sin^2 \frac{d(x_{n+1}, z_J^u)\sqrt{\kappa}}{2} & \le & (1 - \mu_{n+1})\sin^2 \frac{d(x_n,z_J^u)\sqrt{\kappa}}{2}\\
&& + \mu_{n+1} \max\left\{\frac{\gamma_n^J}{\cos (M\sqrt{\kappa})},0\right\} + \Delta^*_{K_0}\left(\sin^2\frac{\varepsilon\sqrt{\kappa}}{4},g\right).
\eua

Letting $\ds t=\frac1{K_0+1}$ in Proposition \ref{prop-gammant-limsup-0}.\eqref{item-limsup-gamma-nt-effective}, we get that 
\[\gamma_n^{K_0}\leq \frac{\cos(M \sqrt{\kappa})}{6}\sin^2\frac{\varepsilon\sqrt{\kappa}}{4},\] 
for all $\ds n \ge \chi^*_{K_0}\left(\frac{1}{3}\sin^2\frac{\varepsilon\sqrt{\kappa}}{4}\right)=
\chi_{K_0}\left(\frac{\cos(M \sqrt{\kappa})}{6}\sin^2\frac{\varepsilon\sqrt{\kappa}}{4}\right)$. 
Thus, 
\bua \gamma_n^{J} & \le & \gamma_n^{K_0} + \frac{\cos (M\sqrt{\kappa})}{6}\sin^2 \frac{\varepsilon\sqrt{\kappa}}{4}\leq 
\frac{\cos(M \sqrt{\kappa})}{3}\sin^2\frac{\varepsilon\sqrt{\kappa}}{4},
\eua
and so,  $\ds \max\left\{\frac{\gamma_n^J}{\cos (M\sqrt{\kappa})},0\right\} \le \frac{1}{3}\sin^2\frac{\varepsilon\sqrt{\kappa}}{4}$ for 
all $\ds n \ge \chi^*_{K_0}\left(\frac{1}{3}\sin^2\frac{\varepsilon\sqrt{\kappa}}{4}\right)$.

Furthermore, by Lemma \ref{mun-divergent-lambdan}, we have that $\ds\sum_{n=1}^\infty \mu_{n+1}=\infty$ with rate of divergence 
$\ds \tilde{\theta}(n) = \theta\left(\left\lceil\frac{1}{ \cos(M \sqrt{\kappa})}\right\rceil n\right)$. Hence, we can apply  Lemma \ref{quant-Aoyama-all-meta-rate-0} with 
\bua 
s_n=\sin^2 \frac{d(x_n, z_J^u)\sqrt{\kappa}}{2}, & \ds t_n= \max\left\{\frac{\gamma_n^J}{\cos (M\sqrt{\kappa})},0\right\}, &
\alpha_n=\mu_{n+1}, \\
\eps= \sin^2\frac{\varepsilon\sqrt{\kappa}}{4}, & \ds \Delta=\Delta^*_{K_0}\left(\sin^2\frac{\varepsilon\sqrt{\kappa}}{4},g\right), & \ds L= \sin^2 \frac{M\sqrt{\kappa}}{2}.
\eua
By letting  $\ds N =\Theta_{K_0}\left(\sin^2\frac{\varepsilon\sqrt{\kappa}}{4}\right)$, it follows that for all  $n \in [N, N+g(N)]$,
\[\sin^2 \frac{d(x_n, z_J^u)\sqrt{\kappa}}{2} \le \sin^2 \frac{\varepsilon\sqrt{\kappa}}{4}, \quad \mbox{and so} \quad d(x_n,z_J^u) \le \frac{\varepsilon}{2}.\]
Obviously, $d(x_n,x_m) \le \varepsilon$ for all $m,n\in [N, N+g(N)]$.  One can easily see that $\ds N\leq \Sigma(\varepsilon,g)$.
\hfill $\Box$

\mbox{ } 

\noindent
{\bf Acknowledgements:} \\[1mm] 
Lauren\c tiu Leu\c stean was supported by a grant of the Romanian 
National Authority for Scientific Research, CNCS - UEFISCDI, project 
number PN-II-ID-PCE-2011-3-0383. \\[1mm]
Adriana Nicolae was supported by a grant of the Romanian
Ministry of Education, CNCS - UEFISCDI, project number PN-II-RU-PD-2012-3-0152.

\end{document}